\documentclass[]{interact}

\usepackage{epstopdf}
\usepackage{subfigure}
\usepackage{multirow}
\usepackage{xcolor}
\usepackage{float}
\usepackage[numbers,sort&compress]{natbib}
\bibpunct[, ]{[}{]}{,}{n}{,}{,}
\newcommand{\f}{\operatorname}

\theoremstyle{plain}
\newtheorem{theorem}{Theorem}[section]
\newtheorem{lemma}[theorem]{Lemma}

\newtheorem{proposition}[theorem]{Proposition}

\theoremstyle{definition}

\theoremstyle{remark}

\begin{document}

\articletype{ }

\title{The Inverse Weighted Lindley Distribution: Properties, Estimation and an Application on a Failure Time Data}


\author{Pedro L. Ramos$^{1}$\footnote{Corresponding author: Pedro Luiz Ramos, Email: pedrolramos@usp.br} \
Francisco Louzada$^{1}$ Taciana K.O. Shimizu$^{1}$ Aline O. Luiz$^{1}$\vspace{0.5cm}\\  $^{1}$Institute of Mathematical Science and Computing, University of S\~ao Paulo, S\~ao Carlos, Brazil}

\maketitle

\begin{abstract}
In this paper a new distribution is proposed. This new model provides more flexibility to modeling data with upside-down bathtub hazard rate function. A significant account of mathematical properties of the new distribution is presented. The maximum likelihood estimators for the parameters in the presence of complete and censored data are presented. Two corrective approaches are considered to derive modified estimators that are bias-free to second order. A numerical simulation is carried out to examine the efficiency of the bias correction. Finally, an application using a real data set is presented in order to illustrate our proposed distribution.
\end{abstract}


\section{Introduction}

In recent years, several new distributions have been introduced
in literature for describing real problems. An important distribution was presented by Lindley \cite{lindley} in the context of fiducial statistics and Bayes' theorem. Ghitany et al. \cite{ghitany1} argued that the Lindley distribution provides flexible mathematical properties and outlined that in many cases this distribution outperforms the exponential distribution. Since then, new generalizations of Lindley distribution have been proposed such as the generalized Lindley \cite{zakerzadeh2009generalized}, extended Lindley \cite{bakouch2012extended}, and Power Lindley \cite{ghitany2013power} distribution.

The study of weight distributions provide new comprehension of standard distributions and contributes in adding more flexibility for fitting data \cite{patil}. Ghitany et al. \cite{ghitany2} presented a two-parameter weighted Lindley (WL) distribution which has bathtub and increasing hazard rate. The WL distribution has probability density function (PDF) given by
\begin{equation}\label{fdpwld}
f(t|\phi,\lambda)=\frac{\lambda^{\phi+1}}{(\phi+\lambda)\Gamma(\phi)}t^{\phi-1}(1+t)e^{^-\lambda t},
\end{equation}
for all $t>0$, $\phi>0$ and $\lambda>0$ where $\Gamma(\phi)=\int_{0}^{\infty}{e^{-x}x^{\phi-1}dx}$ is the gamma function. Mazucheli et al. \cite{mazucheli} compared the finite sample properties of the parameters of the WL distribution numerical simulations using four methods. Wang and Wang \cite{wang} presented bias-corrected MLEs and argued that the proposed estimators are strongly recommended over other estimators without bias-correction. Ali \cite{ali} considered a Bayesian approach and derived several informative and noninformative priors under different loss functions. Ramos and Louzada \cite{ramoslouzada2016} introduced three parameters generalized weighted Lindley distribution. 

In this study, a new two-parameter distribution with upside-down bathtub hazard rate is proposed, hereafter, inverse weighted Lindley (IWL) distribution. This new model can be rewritten as the inverse of the WL distribution. A significant account of mathematical properties for the IWD distribution is presented such as moments, survival properties and entropy functions. The maximum likelihood estimators of the parameters and its asymptotic properties are obtained. Further, two corrective approaches are discussed to derive modified MLEs that are bias-free to second order. The first has an analytical expression derived by Cox and Snell (\ref{coxsnell}) and the second is based on the bootstrap resampling method (see Efron \cite{efron} for more details), which  can be used to reduce bias. Similar corrective approaches has been considered by many authors for other distributions, e.g., Cordeiro et al. \cite{cordeiro2}, Lemonte \cite{lemonte}, Teimouri and Nadarajah \cite{teimouri}, Giles et al. \cite{giles}, Ramos et al. \cite{ramos2016efficient}, Schwartz et al. \cite{schwartz} and Reath et al. \cite{reath2016improved}. In addition, the MLEs in the presence of randomly censored data is presented. Approximated bias-corrected MLEs for censored data are also discussed. A numerical simulation is performed to examine the effect of the bias corrections in the MLEs for complete and censored data. 

The new distribution is a useful generalization of the inverse Lindley distribution \cite{sharma} and can be represented by a two-component mixture model. Mixture models play an important role in statistics for describing heterogeneity (see, Aalen \cite{aalen1988heterogeneity}). Therefore, the IWL distribution can be used to describe data sets in the presence of heterogeneity. For instance, we can be interested in describing the lifetime of components that are composed of new and repaired products, however, only the failure time is observed and the groups are latent variables. In this case, the proposed distribution, as a mixture distribution, can express the heterogeneity in the data. In reliability, this model may be used to describe the lifetime of components associated with a high failure rate
after short repair time. In studies involving the lifetime of patients this model can be useful to describe the course of a disease, where their mortality rate reaches a peak and then declines as the time increase, i.e., problems where their hazard function has upside-down bathtub shape.

In order to illustrate our proposed methodology, we considered a real data set related to failure time of devices of an airline company. Such study is important in order to prevent customer dissatisfaction and customer attrition, and consequently to avoid customer loss. In this context, the choice of the distribution that fits better this data is fundamental for the company reduces its costs. We showed that the inverse weighted Lindley distribution fits better than other well-known distributions for this data set. 

The paper is organized as follows. Section 2 introduces the inverse weighted Lindley distribution. Section 3 presents the properties of the IWL distribution such as moments, survival properties and entropy. Section 4 discusses the inferential procedure based on MLEs for complete and censored data. A bias correction approach is also presented for complete and censored data. Section 5 describes two corrective approaches to reduce the bias in the MLEs for complete and censored data. Section 6 presents a simulation study to verify the performance of the proposed estimators. Section 7 illustrates the relevance of our proposed methodology in a real lifetime data. Section 8 summarizes the present study.

\section{Inverse Weighted Lindley distribution}

A non-negative random variable T follows the IWL distribution with parameters $\phi>0$ and $\lambda>0$ if its PDF is given by
\begin{equation}\label{fdpIWL} 
f(t|\phi,\lambda)=\frac{\lambda^{\phi+1}}{(\phi+\lambda)\Gamma(\phi)}t^{-\phi-1}\left(1+\frac{1}{t}\right)e^{-\lambda t^{-1}}.
\end{equation}

Note that if $\phi=1$, the IWL distribution reduces to the inverse Lindley distribution \cite{sharma}. The IWL distribution can be expressed as a two-component mixture
\[
f(t|\phi,\lambda)= pf_1(t|\phi,\lambda)+(1-p)f_2(t|\phi,\lambda),
\]
where  $p=\lambda/(\lambda+\phi)$ and  $T_j\sim\f{IG}(\phi+j-1,\lambda)$, for $j=1,2$, i.e., $f_j(t|\lambda,\phi)$ is Inverse Gamma distribution, given by
\[
f_j(t|\phi,\lambda)=\frac{\lambda^{\phi+j-1}}{\Gamma(\phi+j-1)}t^{-\phi-j}e^{-\lambda t^{-1}}.
\]

Therefore, the IWL distribution is a mixture distribution and can express the heterogeneity in the data.

\begin{proposition}
Let $T\sim\f{IWL}(\phi,\lambda)$ then $X=1/T$ follows a weighted Lindley distribution \cite{ghitany2}. 
\end{proposition}
\begin{proof} Define the transformation $X = g(T) = \frac{1}{T}$ then the resulting transformation is
	\[
	\begin{aligned}
	f_X(x)&= f_T \left( g^{-1}(x) \right) \left| \frac{d}{dx} g^{-1}(x) \right|
	=\frac{\lambda^{\phi+1}}{(\phi+\lambda)\Gamma(\phi)}{x}^{\phi+1}\left(1+x\right)e^{-\lambda x}\frac{1}{x^2}
	\\&=\frac{\lambda^{\phi+1}}{(\phi+\lambda)\Gamma(\phi)}{x}^{\phi-1}\left(1+x\right)e^{-\lambda x}.
	\end{aligned}
	\]
\end{proof}

Figure \ref{frisIWL} gives examples from the shapes of the density function for different values of $\phi$ and $\lambda$.
\begin{figure}[!htb]
	\centering
	\includegraphics[scale=0.6]{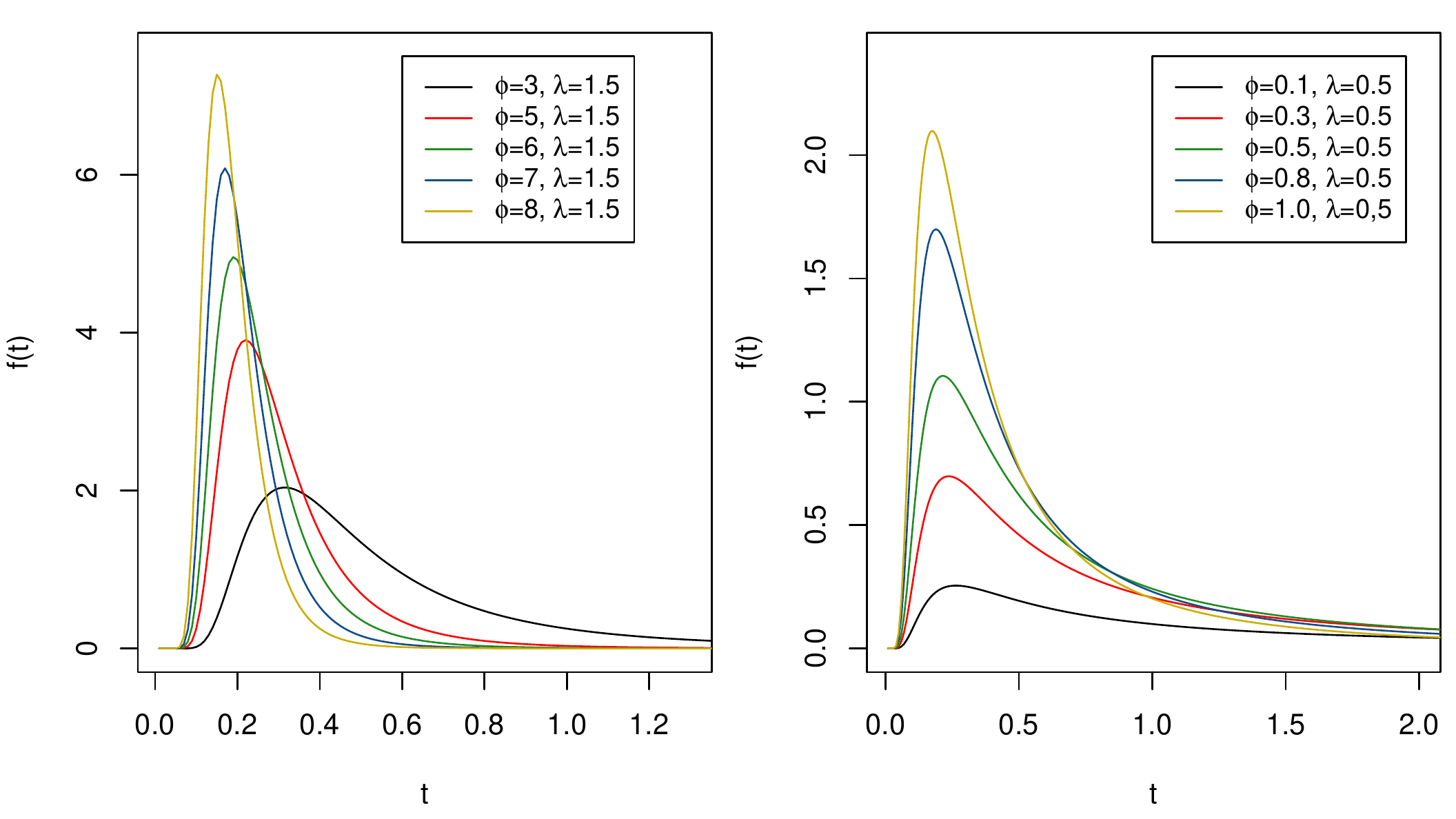}
	\caption{Density function shapes for IWL distribution and considering different values of $\phi$ and $\lambda$.}\label{frisIWL}
\end{figure}

The cumulative distribution function from the IWL distribution is given by
\[
F(t|\phi,\lambda) = \frac{\Gamma\left(\phi,\lambda t^{-1}\right)(\lambda+\phi)+(\lambda t^{-1})^{\phi}e^{-\lambda t^{-1}}}{(\lambda+\phi)\Gamma(\phi)},\ 
\]
where $\Gamma(x,y)=\int_{x}^{\infty}{w^{y-1}e^{-x}dw}$ is the upper incomplete gamma.

\newpage

\section{Properties of IWL Distribution }

In this section, we provide a significant account of mathematical properties of the new distribution.

\subsection{Moments}

Moments play an important role in statistics. They can be used in many applications, for instance the first moment of the PDF is the well know mean, while the second moment is used to obtain the variance, skewness and kurtosis are also obtained from the moments. In the following, we will derive the moments for the IWL distribution.

\begin{proposition}\label{moments1} For the random variable $T$ with $\f{IWL}$ distribution,
	the r-th moment is given by
	\begin{equation}\label{rmIWL}
	\mu_r= E[T^r]=\frac{\lambda^r(\phi+\lambda-r)}{(\lambda+\phi)(\phi-1)(\phi-2)\cdots(\phi-r)}\, ,  \quad \mbox{where} \quad \phi>r .
	\end{equation}
\end{proposition}
\begin{proof}Note that if $W\sim\f{IG}(\phi,\lambda)$ distribution then the r-th moment from the random variable $W$ is given by
	\[
	E_{(\phi,\lambda)}[W^r]=\frac{\lambda^r\Gamma(\phi-r)}{\Gamma(\phi)}=\frac{\lambda^r}{(\phi-1)(\phi-2)\ldots(\phi-r)}\, , \quad \mbox{where} \quad \phi>r.
	\]
	
	Since the IWL distribution can be expressed as a two-component mixture, we have
	\[
	\begin{aligned}
	\mu_r= E[T^r]&=  \int_{0}^{\infty}t^rf(t|\phi,\lambda)dt 
	= pE_{(\phi,\lambda)}[W^r]+(1-p)E_{(\phi+1,\lambda)}[W^r] \\& = \frac{\lambda}{(\lambda+\phi)}\frac{\Gamma(\phi-r)}{\Gamma(\phi)}+\frac{\phi}{(\lambda+\phi)}\frac{\Gamma(\phi+1-r)}{\Gamma(\phi+1)}= \frac{\lambda^r(\lambda+\phi-r)\Gamma(\phi-r)}{(\lambda+\phi)\Gamma(\phi)} \\&=\frac{\lambda^r(\phi+\lambda-r)}{(\lambda+\phi)(\phi-1)(\phi-2)\cdots(\phi-r)}\, ,  \quad \mbox{where} \quad \phi>r .
	\end{aligned}
	\]
\end{proof}

\begin{proposition} The r-th central moment for the random variable $T$ is given by
	\begin{equation}\label{rcmIWLp}
	\begin{aligned} 
	M_r&= E[T-\mu]^r= \sum_{i=0}^{r}\binom{r}{i}(-\mu)^{r-i}E[T^i] \\ &= \sum_{i=0}^{r}\binom{r}{i}\left(-\frac{\lambda(\phi+\lambda-1)}{(\lambda+\phi)(\phi-1)}\right)^{r-i}\left(\frac{\lambda^i(\phi+\lambda-i)}{(\lambda+\phi)(\phi-1)(\phi-2)\cdots(\phi-i)}\right) .
	\end{aligned} 
	\end{equation}
\end{proposition}
\begin{proof} The result follows directly from the proposition \ref{moments1}.\end{proof}

\begin{proposition} A random variable $T$ with $\f{IWL}$ distribution,
	has the mean and variance given by
	\begin{equation*}
	\mu=\frac{\lambda(\phi+\lambda-1)}{(\lambda+\phi)(\phi-1)}, 
	\end{equation*}
	\begin{equation*}
	\sigma^2=\frac{\lambda^2\left((\phi+\lambda-2)(\phi-1)-{(\phi+\lambda-1)}^2(\phi-2) \right)}{(\lambda+\phi)(\phi-2){(\phi-1)}^2}.
	\end{equation*}
\end{proposition}
\begin{proof} From (\ref{rmIWL}) and considering $r=1$, it follows that $\mu_1=\mu$. The second result follows from (\ref{rcmIWLp}) considering $r=2$ and with some algebraic operation the proof is completed.
\end{proof}

\subsection{Survival Properties}

Survival analysis has become a popular branch of statistics with wide range of applications. Although many functions related to survival analysis can be derived for this model, in this section we will present the most common functions. The survival function of IWL distribution representing the probability of an observation does not fail until a specified time $t$ is given by
\[
S(t|\phi,\lambda)=\frac{\gamma\left(\phi,\lambda t^{-1}\right)(\lambda+\phi)-(\lambda t^{-1})^{\phi}e^{-\lambda t^{-1}}}{(\lambda+\phi)\Gamma(\phi)} ,
\]
where  $\gamma(y,x)=\int_{0}^{x}{w^{y-1}e^{-w}}dw$ is the lower incomplete gamma function.   The hazard function of $T$ is given by
\begin{equation}\label{fhwl} 
h(t|\phi,\lambda)=\frac{\lambda^{\phi+1}t^{-\phi-1}\left(1+t^{-1}\right)e^{-\lambda t^{-1}}}{\gamma\left(\phi,\lambda t^{-1}\right)(\lambda+\phi)-(\lambda t^{-1})^{\phi}e^{-\lambda t^{-1}}}\, . 
\end{equation}

This model has upside-down bathtub hazard rate. The following Lemma is useful to prove such result. 

\begin{lemma} Glaser \cite{glaser}: Let T be a non-negative continuous random variable with twice differentiable PDF $f(t)$, hazard rate function $h(t)$ and $\eta(t) =-\frac{\partial}{\partial t}\log f(t)$. Then if	$\eta(t)$ has an upside-down bathtub shape, $h(t)$ has an upside-down bathtub shape.
\end{lemma}

\begin{theorem} The hazard function (\ref{fhwl}) is upside-down bathtub for all $\phi>0$ and $\lambda>0$.
	\begin{proof}
		For IWL distribution we have 
		\[
		\eta(t)=\frac{\phi}{t}+\frac{2}{t}-\frac{1}{(t+1)}-\frac{\lambda}{t^2}\, ,
		\]
		it follows that
		\[
		\eta'(t)=-\frac{\phi}{t^2}-\frac{2}{t^2}+\frac{1}{{(t+1)}^2}+\frac{2\lambda}{t^3}.
		\]
		
		The study of the behaviour of $\eta'(t)$ is not simple. However using the Wolfram$|$Alpha software, we can check that for all $\phi>0$ and $\lambda>0$, $\eta'(t)$ is increasing in $(0,\xi(t|\phi,\lambda))$ and decreasing in $(\xi(t|\phi,\lambda),\infty)$, i.e., $\eta'(t)=0$ at $\xi(t|\phi,\lambda)$, where $\xi(t|\phi,\lambda)$ is a very large function computed to the Wolfram$|$Alpha (available upon request). Therefore, $\eta(t)$ and consequently $h(t)$ has upside-down bathtub shape.
	\end{proof} 
\end{theorem}

This properties make the IWL distribution an useful model for reliability data. Figure 2 gives examples of different shapes for the hazard function.
\begin{figure}[!htb]
	\centering
	\includegraphics[scale=0.5]{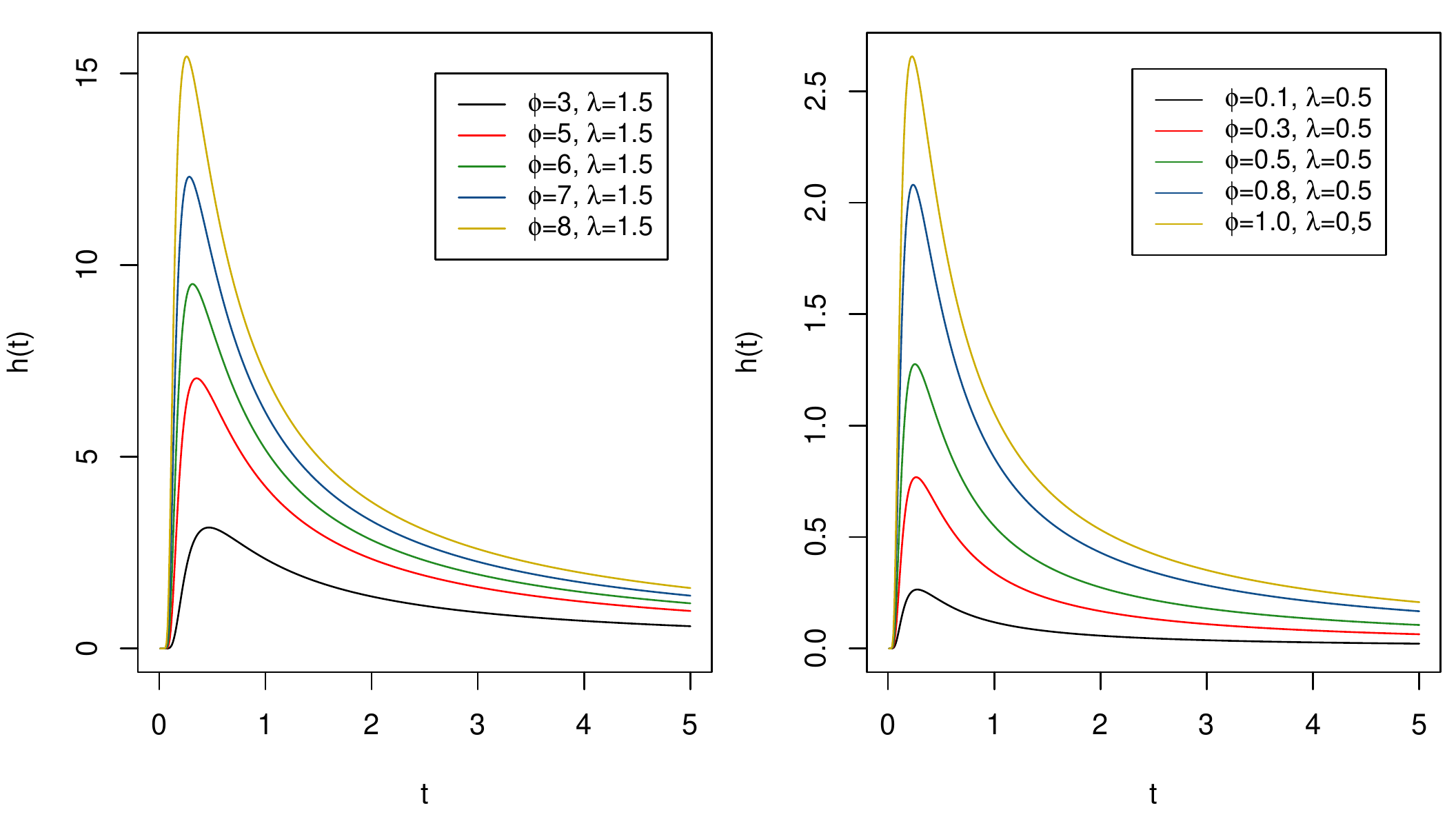}
	\caption{Hazard function shapes for IWL distribution and considering different values of $\phi$ and $\lambda$.}\label{friswl}
\end{figure}

\begin{proposition} The mean residual life function $r(t|\phi,\lambda)$ of the $\f{IWL}$ distribution is given by
	\[
	\begin{aligned} 
	r(t|\phi,\lambda)&= \frac{1}{S(t)}\int_{t}^{\infty}yf(y|\lambda,\phi)dy-t  =\frac{\lambda\gamma\left(\phi,\lambda t^{-1}\right)+\lambda^2 \gamma\left(\phi,\lambda t^{-1}\right)}{\gamma\left(\phi,\lambda t^{-1}\right)(\lambda+\phi)-(\lambda t^{-1})^{\phi}e^{-\lambda t^{-1}}}-t.
	\end{aligned} 
	\]
\end{proposition}
\begin{proof} Note that, for the Inverse Gamma distribution we have that
	\begin{equation*} 
	\int_{t}^{\infty}yf_j(y|\phi,\lambda)dy=\frac{\lambda}{\Gamma(\phi+j-1)}\gamma\left[\phi+j-2,\lambda t^{-1}\right],  \ \ j=1,2 .
	\end{equation*}
	Using the following relationship 
	\begin{equation*}
	r(t|\phi,\lambda)=\frac{1}{S(t)}\left[p\int_{t}^{\infty}yf_1(y|\lambda,\phi)dy+(1-p)\int_{t}^{\infty}yf_2(y|\lambda,\phi)dy\right]-t\, ,
	\end{equation*}
	and after some algebraic manipulations, the proof is completed.
\end{proof}

\subsection{Entropy}

In information theory, entropy has played a central role as a measure of the uncertainty associated with a random variable. Shannon's entropy is one of the most important metrics in information theory. The Shannon's Entropy from IWL distribution is given by solving the following equation
\begin{equation}\label{she1} 
H_S(\phi,\lambda)=-\int_{0}^{\infty}\log\left(\frac{\lambda^{\phi+1}}{(\phi+\lambda)\Gamma(\phi)}t^{-\phi-1}\left(1+\frac{1}{t}\right)e^{-\lambda t^{-1}}\right)f(t|\phi,\lambda)dt.
\end{equation}	

\begin{proposition} A random variable $T$ with $\f{IWL}$ distribution,
	has the Shannon's Entropy given by
	\begin{equation*}
	\begin{aligned}
	H_S(\phi,\lambda)=&\log(\lambda+\phi)+\log\Gamma(\phi)+\frac{\phi(\lambda+\phi+1)}{(\lambda+\phi)}-(\phi+1)\left(\frac{1}{\lambda+\phi} +\psi(\phi) \right) \\ & - \frac{\lambda^{\phi+1}\Omega(\phi,\lambda)}{(\lambda+\phi)\Gamma(\phi)}.
	\end{aligned}
	\end{equation*}
	where $\Omega(\phi,\lambda)=\int_{0}^{\infty}(x+1)\log(x+1)x^(\phi-1)e^{-\lambda x}dx$.
	
	\begin{proof} From the equation (\ref{she1}) we have
		\begin{equation*}\label{shep1} 
		H_S(\phi,\lambda)= \ (\phi+1)\log\lambda-\log(\lambda+\phi)-\log\Gamma(\phi)-\lambda E\left[t_i^{-1}\right] -(\phi+1)E[\log(t_i)] +E\left[\log(1+t_i^{-1})\right].
		\end{equation*}
		Since
		\begin{equation*}
		E[\log(t)]=\log(\lambda)-\frac{1}{\lambda+\phi} -\psi(\phi), \quad \mbox{and }
		\end{equation*}
		\begin{equation*}
		E\left[t_i^{-1}\right]=\frac{(\phi+1)}{\lambda}-\frac{1}{\lambda+\phi}=\frac{\phi(\lambda+\phi+1)}{\lambda(\lambda+\phi)}.
		\end{equation*}
		\begin{equation*}
		E\left[\log(1+t_i^{-1})\right]=\frac{\lambda^{\phi+1}}{(\lambda+\phi)\Gamma(\phi)}\int_{0}^{\infty}(x+1)\log(x+1)x^(\phi-1)e^{-\lambda x}dx.
		\end{equation*}
		Then
		\begin{equation*}
		\begin{aligned}
		H_S(\phi,\lambda)=& \ (\phi+1)\left(\frac{1}{\lambda+\phi} +\psi(\phi) \right)-\log(\lambda+\phi)-\log\Gamma(\phi)-\frac{\phi(\lambda+\phi+1)}{(\lambda+\phi)} \\ & - \frac{\lambda^{\phi+1}\Omega(\phi,\lambda)}{(\lambda+\phi)\Gamma(\phi)}.
		\end{aligned}
		\end{equation*}
	\end{proof}
\end{proposition}

\section{Inference}

In this section, we present the maximum likelihood estimator  of the parameters $\phi$ and $\lambda$ of the IWL distribution. Additionally, MLEs considering randomly censored data are also discussed.

\subsection{Maximum Likelihood Estimation}

Among the statistical inference methods, the maximum likelihood method is widely used due to its better asymptotic properties. Under the maximum likelihood method, the estimators are obtained from maximizing the likelihood function. Let $T_1,\ldots,T_n$ be a random sample such that $T\sim \f {IWL}(\phi,\mu)$. In this case, the likelihood function from (\ref{fdpIWL}) is given by
\begin{equation*}\label{veroiIWL}
L(\boldsymbol{\theta};\boldsymbol{t})=\frac{\lambda^{n(\phi+1)}}{{(\phi+\lambda)}^n{\Gamma(\phi)}^n}\left\{\prod_{i=1}^n{t_i^{-\phi-1}}\right\}\prod_{i=1}^n\left(1+\frac{1}{t_i}\right)\exp\left\{-\lambda\sum_{i=1}^n \frac{1}{t_i}\right\}.
\end{equation*}

The log-likelihood function $l(\boldsymbol{\theta};\boldsymbol{t})=\log{L(\boldsymbol{\theta};\boldsymbol{t})}$ is given by
\begin{equation}\label{verogIWL2}
l(\boldsymbol{\theta};\boldsymbol{t})=\ n(\phi+1)\log\lambda-n\log(\lambda+\phi)-n\log\Gamma(\phi)-\lambda\sum_{i=1}^n \frac{1}{t_i} -(\phi+1)\sum_{i=1}^{n}\log(t_i).
\end{equation}

From the expressions $\frac{\partial}{\partial \phi}l(\boldsymbol{\theta};\boldsymbol{t})=0$, $\frac{\partial}{\partial \lambda}l(\boldsymbol{\theta};\boldsymbol{t})=0$, we get the likelihood equations
\begin{equation*}\label{verogg21} 
n\log(\lambda)-\sum_{i=1}^{n}\log(t_i)-\frac{n}{\lambda+\phi} -n\psi(\phi)=0 \, ,
\end{equation*}
\begin{equation*}\label{verogg22} 
\frac{n(\phi+1)}{\lambda}-\sum_{i=1}^{n}\frac{1}{t_i}-\frac{n}{\lambda+\phi}=0 \, ,
\end{equation*}
where $\psi(k)=\frac{\partial}{\partial k}\log\Gamma(k)=\frac{\Gamma'(k)}{\Gamma(k)}$ is the digamma function. After some algebraic manipulation the solution of $\lambda_{MLE}$ is given by
\begin{equation*}\label{verogg23} 
\hat\lambda_{MLE}=\frac{-\hat\phi_{MLE}\left(\xi(\boldsymbol{t})-1\right)+\sqrt{\left(\hat\phi_{MLE}\left(\xi(\boldsymbol{t})-1\right) \right)^2+4\,\xi(\boldsymbol{t})\left(\hat\phi_{MLE}^2+\hat\phi_{MLE}\right)}}{2\xi(\boldsymbol{t})} \, ,
\end{equation*}
where $\xi(\boldsymbol{t})=\sum_{i=1}^{n}(nt_i)^{-1}$
and $\hat\phi_{MLE}$ can be obtained solving the nonlinear system
\begin{equation}\label{verogg24} 
n\log(\hat\lambda_{MLE})-\sum_{i=1}^{n}\log(t_i)-\frac{n}{\hat\lambda_{MLE}+\hat\phi_{MLE}} -n\psi(\hat\phi_{MLE})=0 .
\end{equation}

These results are a simple modification of the results obtained for Ghitany et al. \cite{ghitany1} for the WL distribution. Under mild conditions the ML estimates are asymptotically normal distributed with a bivariate normal distribution given by
\begin{equation*} (\hat\phi,\hat\lambda) \sim N_2[(\phi,\lambda),I^{-1}(\phi,\lambda)] \mbox{ for } n \to \infty , \end{equation*}
where the elements of the Fisher information matrix I$(\phi,\lambda)$ are given by
\begin{equation*}
h_{11}(\phi,\lambda)=-\frac{n}{{(\lambda+\phi)}^2}+n\psi'(\phi)\, ,
\end{equation*}
\begin{equation*}
h_{12}(\phi,\lambda)=h_{21}(\phi,\lambda)=-\frac{n}{\lambda}-\frac{n}{{(\lambda+\phi)}^2}\, ,
\end{equation*}
\begin{equation*}
h_{22}(\phi,\lambda)=\frac{n(\phi+1)}{\lambda^2}-\frac{n}{{(\lambda+\phi)}^2}\, ,
\end{equation*}
and $\psi'(k)=\frac{\partial}{\partial^2 k}\log\Gamma(k)$ is the trigamma function. An interesting property of the IWL distribution is that the observed matrix information is equal to the expected information matrix.

\subsection{Random Censoring}

In survival analysis and industrial lifetime testing, random censoring schemes have been received special attention. Suppose that the $i$th  individual has a lifetime $T_i$ and a censoring time $C_i$, moreover the random censoring times $C_i$s are independent of $T_i$s and their distribution does not depend on the parameters, then the data set is $(t_i,\delta_i)$, where $t_i=\min(T_i,C_i)$ and $\delta_i=I(T_i\leq C_i)$. This type of censoring have as special case the type I and II censoring mechanism. The likelihood function for $\boldsymbol{\theta}$ is given by
\begin{equation*}\label{eqveroc2} L(\boldsymbol{\theta,t})=\prod_{i=1}^n f(t_i|\boldsymbol{\theta})^{\delta_i}S(t_i|\boldsymbol{\theta})^{1-\delta_i} .\end{equation*}

Let $T_1,\cdots,T_n$ be a random sample of IWL distribution, the likelihood function considering data with random censoring is given by 
\begin{equation}\label{eqverodwca}
\begin{aligned}
L(\lambda, \phi|\boldsymbol{t})=&\frac{\lambda^{d(\phi+1)}}{(\lambda+\phi)^n\Gamma(\phi)^n}\prod_{i=1}^n\left((\lambda+\phi)\gamma(\phi,\lambda t_{i}^{-1})-\left(\lambda t_{i}^{-1}\right)^{\phi}e^{-\lambda t_{i}^{-1}}\right)^{1-\delta_i}\\ & \times\left(t_{i}^{-\phi-1}(1+t_{i}^{-1})e^{-\lambda t_{i}^{-1}}\right)^{\delta_i} .
\end{aligned}
\end{equation}

The logarithm of the likelihood function (\ref{eqverodwca}) is given by
\begin{equation*}\label{logverctr}
\begin{aligned}
l(\lambda,\phi|\boldsymbol{t})= & \, -(\phi+1)\sum_{i=1}^{n}\delta_i\log(t_{i})-\lambda\sum_{i=1}^{n}\delta_it_{i}^{-1} +d(\phi+1)\log(\lambda)- n\log(\phi+\lambda) \\ & +\sum_{i=1}^{n}(1-\delta_i)\log\left((\lambda+\phi)\gamma(\phi,\lambda t_{i}^{-1})-{(\lambda t_{i}^{-1})}^\phi e^{-\lambda t_{i}^{-1}}\right)- n\log\left(\Gamma(\phi)\right)\\ & +\sum_{i=1}^{n}\delta_i\log(1+t_{i}^{-1}) .
\end{aligned}
\end{equation*}

From ${\partial}l(\lambda, \phi|\boldsymbol{t})/{\partial \lambda}=0$ and ${\partial}l(\lambda, \phi|\boldsymbol{t})/{\partial \phi}=0$, the likelihood equations are given as follows
{\small
\begin{equation*}
\sum_{i=1}^{n}\frac{(1-\delta_i)\left(\gamma(\phi,\lambda t_{i}^{-1})+(\lambda+\phi)\left(\lambda t_{i}^{-1}\right)^{\phi-1}e^{-\lambda t_{i}^{-1}}-\phi\lambda^{\phi-1}t_{i}^{-\phi}e^{-\lambda t_{i}^{-1}}-\left(\lambda t_{i}^{-1}\right)^{\phi+1}e^{-\lambda t_{i}^{-1}}\right)}{\left((\lambda+\phi)\gamma(\phi,\lambda t_{i}^{-1})\right)-\left(\lambda t_{i}^{-1}\right)^{\phi}e^{-\lambda t_{i}^{-1}}} =
\end{equation*} 
\begin{equation}\label{verowl21}
\frac{n}{\lambda+\phi}-\frac{d(\phi+1)}{\lambda}+\sum_{i=1}^{n}\delta_i t_{i}^{-1} ,
\end{equation} }
\begin{equation*}
\sum_{i=1}^{n}\frac{(1-\delta_i)\left(\gamma(\phi,\lambda t_{i}^{-1})+(\lambda+\phi)\Psi(\phi,\lambda t_{i}^{-1}) -\left(\lambda t_{i}^{-1}\right)^{\phi}\log(\lambda t_{i}^{-1})e^{-\lambda t_{i}^{-1}}\right)}{\left((\lambda+\phi)\gamma(\phi,\lambda t_{i}^{-1})\right)-\left(\lambda t_{i}^{-1}\right)^{\phi}e^{-\lambda t_{i}^{-1}}} = -d\log(\lambda) 
\end{equation*}
\begin{equation}\label{verowl22}
+\frac{n}{\lambda+\phi} +n\psi(\phi) +\sum_{i=1}^{n}\delta_i\log(t_{i}^{-1})\, ,
\end{equation}
where $\Psi(k,x)={\partial}\,\gamma(k,x)/{\partial k}$ can be computed numerically. Numerical methods are required in order to find the solution of these non-linear equations.

\section{Bias correction for the maximum likelihood estimators}

In this section, we discuss modified MLEs based on two corrective approaches that are bias-free to second order. Firstly a corrective analytical approach is presented than the bootstrap resampling method is presented. 

\subsection{A corrective approach}\label{secbiascor}

Consider the likelihood function $L(\boldsymbol{\theta};\boldsymbol{t})$ with a $p$-dimensional vector of parameters $\boldsymbol{\theta}$. Thus, the joint cumulants of the derivatives of $l(\boldsymbol{\theta};\boldsymbol{t})$ can be written by
\begin{align*}
h_{ij}(\boldsymbol{\theta})&=E\left(\frac{\partial^2 l(\boldsymbol{\theta};\boldsymbol{t})}{\partial \theta_i\theta_j}\right), \quad
h_{ijl}(\boldsymbol{\theta})=E\left(\frac{\partial^3 l(\boldsymbol{\theta};\boldsymbol{t})}{\partial\theta_i\partial\theta_j\partial\theta_l} \right) \mbox{and} \\ h_{ij,l}(\boldsymbol{\theta})&=E\left(\frac{\partial^2 l(\boldsymbol{\theta};\boldsymbol{t})}{\partial\theta_i\partial\theta_j}.\frac{\partial l(\boldsymbol{\theta};\boldsymbol{t})}{\partial\theta_l} \right), \quad \mbox{for} \quad i,j,l=1,\ldots,p.
\end{align*}
Consequently, the derivatives of such cumulants are given by
\begin{align*}
h_{ij}^{(l)}(\boldsymbol{\theta})&=\dfrac{\partial h_{ij}(\boldsymbol{\theta})}{\partial\theta_l}, \quad \mbox{for} \quad i,j,l=1,\ldots,p.
\end{align*}
The bias of $\theta_m$ studied by Cox and Snell \cite{cox} for independent sample without necessarily be identically distributed can be written by
\begin{equation}\label{coxsnell}
Bias(\hat{\theta}_m)=\sum_{i=1}^{p}\sum_{j=1}^{p}\sum_{k=1}^{p}s_{mi}(\boldsymbol{\theta})s_{jl}(\boldsymbol{\theta})\left(h_{ij,l}(\boldsymbol{\theta})+0.5h_{ijl}(\boldsymbol{\theta}) \right)+O(n^{-2}) \, ,
\end{equation}
where $s^{ij}$ is the $(i,j)$-th element of the inverse of Fisher's information matrix of $\boldsymbol{\hat\theta}$, $K=\{-h_{ij}\}$.
Cordeiro and Klein \cite{cordeiro1} proved that even if the data are dependent the expression (\ref{coxsnell}) can be re-written as 
\begin{equation}\label{coxsnell2}
Bias(\hat{\theta}_m)=\sum_{i=1}^{p}s_{mi}(\boldsymbol{\theta})\sum_{j=1}^{p}\sum_{k=1}^{p}s_{jl}(\boldsymbol{\theta})\left(h_{ij}^{(l)}(\boldsymbol{\theta})-0.5h_{ijl}(\boldsymbol{\theta}) \right)+O(n^{-2}) .
\end{equation}
Let $a_{ij}^{l}=h_{ij}^{(l)}-\frac{1}{2}h_{ij}^{(l)}$ and define the matrix $A=[A^{(1)}|A^{(2)}|\ldots|A^{(p)}]$ with $A^{(l)}=\{a_{ij}^{(l)}\}$, for $i,j,l=1,\ldots,p$. Thus, the expression for the bias of $\boldsymbol{\hat\theta}$ can be expressed as
\begin{equation}\label{coxsnell3}
Bias(\hat{\theta}_m)=K^{-1}A.\mbox{vec}{(K^{-1})}+O(n^{-2}).
\end{equation}
A bias corrected MLE for $\hat{\boldsymbol{\theta}}$ is obtained as
\begin{equation}
\hat{\boldsymbol{\theta}}_{CMLE}=\hat{\boldsymbol{\theta}}-K^{-1}A.\mbox{vec}{(K^{-1})}\, ,
\end{equation}
where $\hat{\boldsymbol{\theta}}$ is the MLE of the parameter $\boldsymbol{\theta}$, $\hat{K}=K|_{\boldsymbol{\theta}=\hat{\boldsymbol{\theta}}}$ and $\hat{A}=A|_{\boldsymbol{\theta}=\hat{\boldsymbol{\theta}}}$. The bias of $\hat{\boldsymbol{\theta}}_{CMLE}$ is unbiased $O(n^{-2})$.
For the IWL distribution the higher-order derivatives can be easily obtained since they do not involve $\boldsymbol{t}$, thus, we have
\begin{align*}
h_{111}(\boldsymbol{\theta})&=h_{11}^{(1)}(\boldsymbol{\theta})=-\frac{2n}{{(\lambda+\phi)}^3}-n\psi''(\phi)\, , \\
h_{122}(\boldsymbol{\theta})&=h_{221}(\boldsymbol{\theta})=h_{212}(\boldsymbol{\theta})=h_{12}^{(2)}(\boldsymbol{\theta})=h_{22}^{(1)}(\boldsymbol{\theta})=-\frac{2n}{{(\lambda+\phi)}^3}-\frac{n}{\lambda^2} , \\
h_{222}(\boldsymbol{\theta})&=h_{22}^{(2)}(\boldsymbol{\theta})=-\frac{2n}{{(\lambda+\phi)}^3}-\frac{2n(\phi+1)}{\lambda^3}\ \ \mbox{ and } 
\end{align*}
\begin{align*}
h_{211}(\boldsymbol{\theta})&=h_{112}(\boldsymbol{\theta})=h_{121}(\boldsymbol{\theta})=h_{12}^{(1)}(\boldsymbol{\theta})=h_{11}^{(2)}(\boldsymbol{\theta})=-\frac{2n}{{(\lambda+\phi)}^3},
\end{align*}
where $\psi''(k)=\frac{\partial}{\partial^3 k}\log\Gamma(k)$. The matrix $K$ is given by
\[K=
\begin{bmatrix}
\frac{n}{(\lambda+\phi)^2}-n\psi'(\phi) & \frac{n}{\lambda}+\frac{n}{(\lambda+\phi)^2} \\
\frac{n}{\lambda}+\frac{n}{(\lambda+\phi)^2} & -\frac{n(\phi+1)}{\lambda^2}+\frac{n}{(\lambda+\phi)^2}
\end{bmatrix} .
\]
To obtain the matrix $A$ of (\ref{coxsnell3}), we present the elements of $A^{(1)}$
\begin{align*}
a_{11}^{(1)}&=h_{11}^{(1)}-\frac{1}{2}h_{111}=-\frac{n}{(\lambda+\phi)^3}-\frac{n\psi''(\phi)}{2}\, , \\
a_{12}^{(1)}&=a_{21}^{(1)}=h_{12}^{(1)}-\frac{1}{2}h_{112}=-\frac{n}{(\lambda+\phi)^3}\, , \\
a_{22}^{(1)}&=h_{22}^{(1)}-\frac{1}{2}h_{221}=-\frac{n}{(\lambda+\phi)^3}-\frac{n}{2\lambda^2}\, ,
\end{align*}
and the elements of $A^{(2)}$ are
\begin{align*}
a_{11}^{(2)}&=h_{11}^{(2)}-\frac{1}{2}h_{112}=-\frac{n}{(\lambda+\phi)^3}\, , \\
a_{12}^{(2)}&=a_{21}^{(2)}=h_{12}^{(2)}-\frac{1}{2}h_{122}=-\frac{n}{(\lambda+\phi)^3}-\frac{n}{2\lambda^2}\, , \\
a_{22}^{(2)}&=h_{22}^{(2)}-\frac{1}{2}h_{222}=-\frac{n}{(\lambda+\phi)^3}-\frac{n(\phi+1)}{\lambda^3}.
\end{align*}
Thus, the matrix $A=\ [A^{(1)}|A^{(2)}]$ is expressed by
\begin{align*}
A= \ n
\begin{pmatrix}
-\frac{1}{(\lambda+\phi)^3}-\frac{\psi''(\phi)}{2} & -\frac{1}{(\lambda+\phi)^3} & -\frac{1}{(\lambda+\phi)^3} & -\frac{1}{(\lambda+\phi)^3}-\frac{1}{2\lambda^2} \\
-\frac{1}{(\lambda+\phi)^3} & -\frac{1}{(\lambda+\phi)^3}-\frac{1}{2\lambda^2} & -\frac{1}{(\lambda+\phi)^3}-\frac{1}{2\lambda^2} & -\frac{1}{(\lambda+\phi)^3}-\frac{(\phi+1)}{\lambda^3}
\end{pmatrix} .
\end{align*}
Finally, the bias-corrected maximum likelihood estimators are given by
\begin{equation}\label{equmatbias}
\begin{pmatrix}
\hat\phi_{CMLE} \\
\hat\lambda_{CMLE}
\end{pmatrix}
=
\begin{pmatrix}
\hat\phi_{MLE} \\
\hat\lambda_{MLE}
\end{pmatrix}
-\hat{K}^{-1}\hat{A}.vec(\hat{K}^{-1}) \, ,
\end{equation}
where $\hat{K}=K|_{\phi=\hat{\phi},\lambda=\hat{\lambda}}$ and $\hat{A}=A|_{\phi=\hat{\phi},\lambda=\hat{\lambda}}$. It is important to point out that, since the higher-order do not involve $\boldsymbol{t}$, they are the same of the WL distribution \cite{wang}.

A bias corrected approach can be considered for censored data. Although the Fisher information matrix related to the MLEs (\ref{eqverodwca}) does not present closed-form expressions, we can consider the bias corrected presented in (\ref{secbiascor}). In this case, approximated bias-corrected maximum likelihood estimates (ACMLE) are archived by
\[
\begin{pmatrix}
\hat\phi_{ACMLE} \\
\hat\lambda_{ACMLE}
\end{pmatrix}
=
\begin{pmatrix}
\hat\phi_{MLE} \\
\hat\lambda_{MLE}
\end{pmatrix}
-\hat{K}^{-1}\hat{A}.vec(\hat{K}^{-1})\, ,
\]
\noindent
where $\hat{K}=K|_{\phi=\hat\phi_{MLE},\lambda=\hat\lambda_{MLE}}$, $\hat{A}=A|_{\phi=\hat\phi_{MLE},\lambda=\hat\lambda_{MLE}}$ and $\hat\phi_{MLE}$ and $\hat\lambda_{MLE}$ are the solutions of $(\ref{verowl21})$ and $(\ref{verowl22})$. However, the bias of $\hat{\theta}_{ACMLE}$ is not an unbiased estimator with $O(n^{-2})$.

	\subsection{Bootstrap resampling method}\label{secboot}

In what follows we consider the bootstrap resampling method proposed by Efron \cite{efron} to reduce the bias of the MLEs. Such method consists in generating pseudo-samples from the original sample to estimate the bias of the MLEs. Thus, the bias-corrected MLEs is given by subtraction of the estimated bias with the original MLEs.

Let $\mathbf{y}=(y_{1},\ldots,y_{n})^\top$ be a sample with $n$ observations randomly selected from the random variable $Y$ in which has the distribution function $F=F_{\nu}(y)$. Thus, let the parameter $\nu$ be a function of $F$ given by $\nu=t(F)$. Finally, let $\hat{\nu}$ be an estimator of $\nu$ based on $\mathbf{y}$, i.e., $\hat{\nu}=s(\mathbf{y})$. The pseudo-samples $\mathbf{y^{*}}=(y^{*}_{1},\ldots,y^{*}_{n})^\top$ is obtained from the original sample $\mathbf{y}$ through resampling with replacement. The bootstrap replicates of $\hat{\nu}$ is calculated, where $\hat{\nu}^{*}=s(\mathbf{y}^{*})$ and the empirical cdf (ecdf) of $\hat{\nu}^{*}$ is used to estimate $F_{\hat{\nu}}$ (cdf of $\hat{\nu}$). Let $B_{F}(\hat{\nu},\nu)$ be the bias of the estimator $\hat{\nu}=s(\mathbf{y})$ given by         
\begin{equation*}
B_{F}(\hat{\nu},\nu)=E_{F}[\hat{\nu},\nu]=E_{F}[s(\mathbf{y})]-\nu(F).
\end{equation*}

Note that the subscript of the expectation $F$ indicates that is taken with respect to $F$. The bootstrap estimators of the bias were obtained by replacing $F$ with $F_{\hat{\nu}}$, where $F$ generated the original sample. Therefore, the bootstrap bias estimate is given by
\begin{equation*}
\hat{B}_{F_{\hat{\nu}}}(\hat{\nu},\nu)=E_{F_{\hat{\nu}}}[\hat{\nu^{*}}]-\hat{\nu}.
\end{equation*}        
If we have $B$ bootstrap samples $(\mathbf{y}^{*(1)},\mathbf{y}^{*(2)},\ldots,\mathbf{y}^{*(B)})$ which are generated independently from the original sample $\mathbf{y}$ and the respective bootstrap estimates $(\hat{\nu}^{*(1)},\hat{\nu}^{*(2)},\ldots,\hat{\nu}^{*(B)})$ are calculated, then it is achievable to determine the bootstrap expectations $E_{F_{\hat{\nu}}}[\hat{\nu^{*}}]$ approximately by
\begin{equation*}
\hat{\nu}^{*(.)}=\frac{1}{B}\sum_{i=1}^{B}\hat{\nu}^{*(i)}.
\end{equation*}

Therefore, the bootstrap bias estimate based on $B$ replications of $\hat{\nu}$ is $\hat{B}_{F}(\hat{\nu},\nu)=\hat{\nu}^{*(.)}-\hat{\nu}$, which results in the bias corrected estimators obtained through by bootstrap resampling method that is given by
\begin{equation*}
\nu^{B}=\hat{\nu}-\hat{B}_{F}(\hat{\nu},\nu)=2\hat{\nu}-\hat{\nu}^{*(.)}.
\end{equation*}   
In our case, we have $\nu^{B}$ denoted by $\hat{\theta}_{BOOT}=(\hat{\phi}_{BOOT},\hat{\lambda}_{BOOT})^\top$.

\section{Simulation Analysis}

In this section a simulation study is presented to compare the efficiency of the maximum likelihood method and the bias correction approaches in the presence of complete and censored data. The proposed comparisons are performed by computing the mean relative errors (MRE) and the relative mean square errors (RMSE) given by
\begin{equation*}
\f{MRE}_i=\frac{1}{N}\sum_{j=1}^{N}\frac{\hat\theta_{i,j}}{\theta_i} \ ,  \quad \f{RMSE}_i=\frac{1}{N}\sum_{j=1}^{N}\frac{(\hat\theta_{i,j}-\theta_i)^2}{\theta_i^2},  \quad \f{for} \ \ i=1,2,
\end{equation*} 
where $N$ is the number of estimates obtained through the MLE, CMLE and the bootstrap approach. The $95\%$ coverage probability of the asymptotic confidence intervals are also evaluated. Considering this approach, we expected that the most efficient estimation method returns the MREs closer to one with smaller RMSEs. Moreover, for a large number of experiments, using a $95\%$ confidence level, the frequencies of intervals that covered the true values of $\boldsymbol{\theta}$ should be closer to $95\%$. Following Reath et al. \cite{reath2016improved} we used B=1,000 for the bootstrap method. The programs can be obtained, upon request. The random sample of the IWL were generated considering the following algorithm:

\begin{enumerate}
	
	\item Generate $U_i\sim \f{Uniform}(0,1), i=1,\ldots,n$;
	
	\item Generate $X_i\sim \f{IG}(\phi,\lambda), i=1,\ldots,n$;
	
	\item Generate $Y_i\sim \f{IG}(\phi+1,\lambda), i=1,\ldots,n$;
	
	\item If $U_i\leq p=\lambda/(\lambda+\phi)$, then set $T_i=X_i$, otherwise, set $T_i=Y_i, i=1,\ldots,n$. 
	
\end{enumerate}

\subsection{Complete Data}

The simulation study is performed considering the values: $\boldsymbol{\theta}=((0.5,2)$,$(2,4))$, $N=30,000$ and $n=(20,25,\ldots$, $130)$. It is important to point out that, similar results were achieved for different choices of $\phi$ and $\lambda$. The uniroot procedure available in R is considered to find the solution of the non-linear equation (\ref{verogg24}). The bias correction is computed directly from (\ref{equmatbias}). Figures \ref{fsimulation1} and \ref{fsimulation2} present the MRE, RMSE and the coverage probability with a $95\%$ confidence level related to the MLE, CMLE and the bootstrap under different values of $n$. 

\begin{figure}[!h]
	\centering
	\includegraphics[scale=0.645]{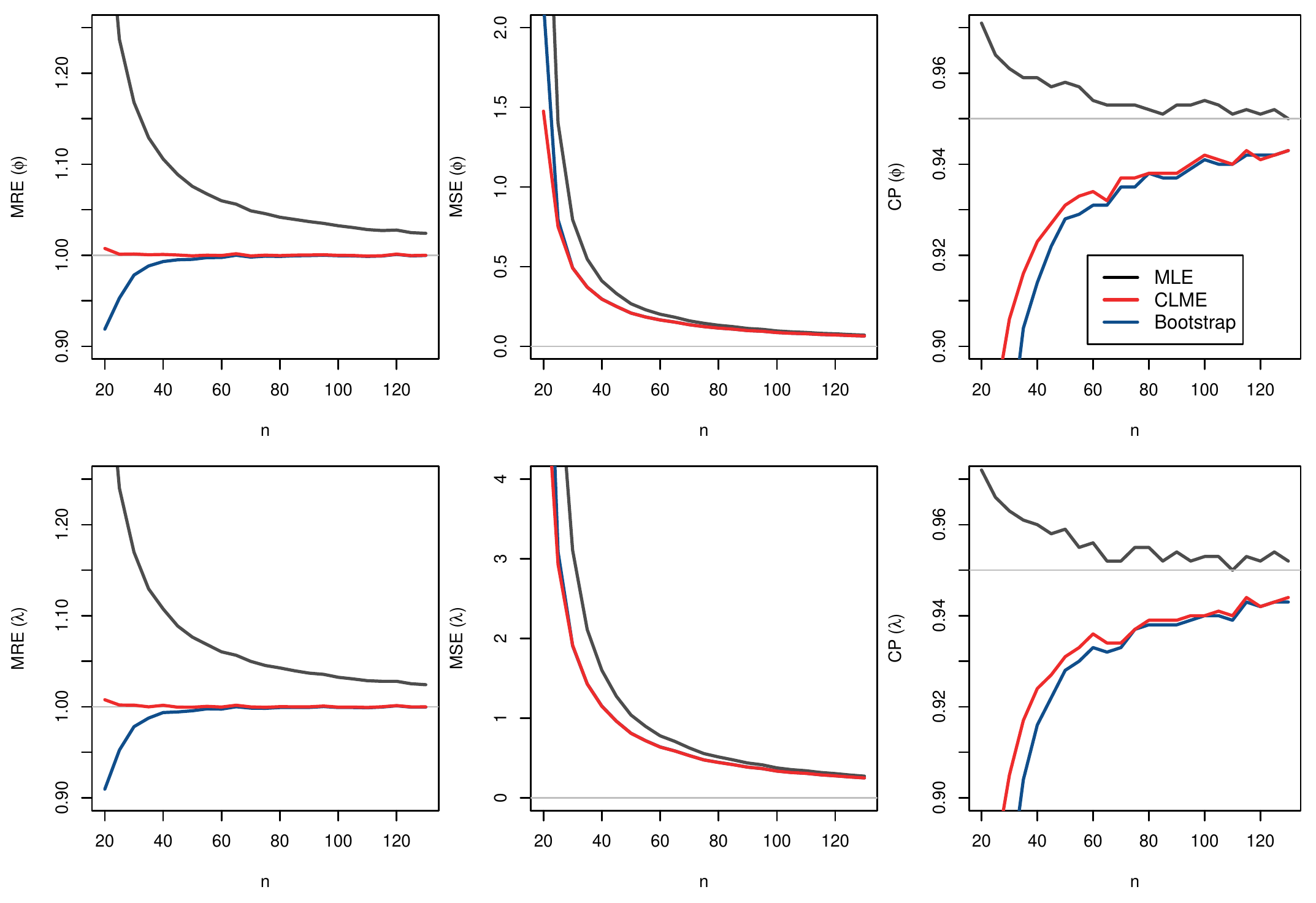}
	\caption{MREs, MSEs related to the estimates of $\lambda=2$ and $\phi=4$ for $N=30,000$ simulated samples, considering different values of $n$.}\label{fsimulation1}
\end{figure}
\begin{figure}[!h]
	\centering
	\includegraphics[scale=0.645]{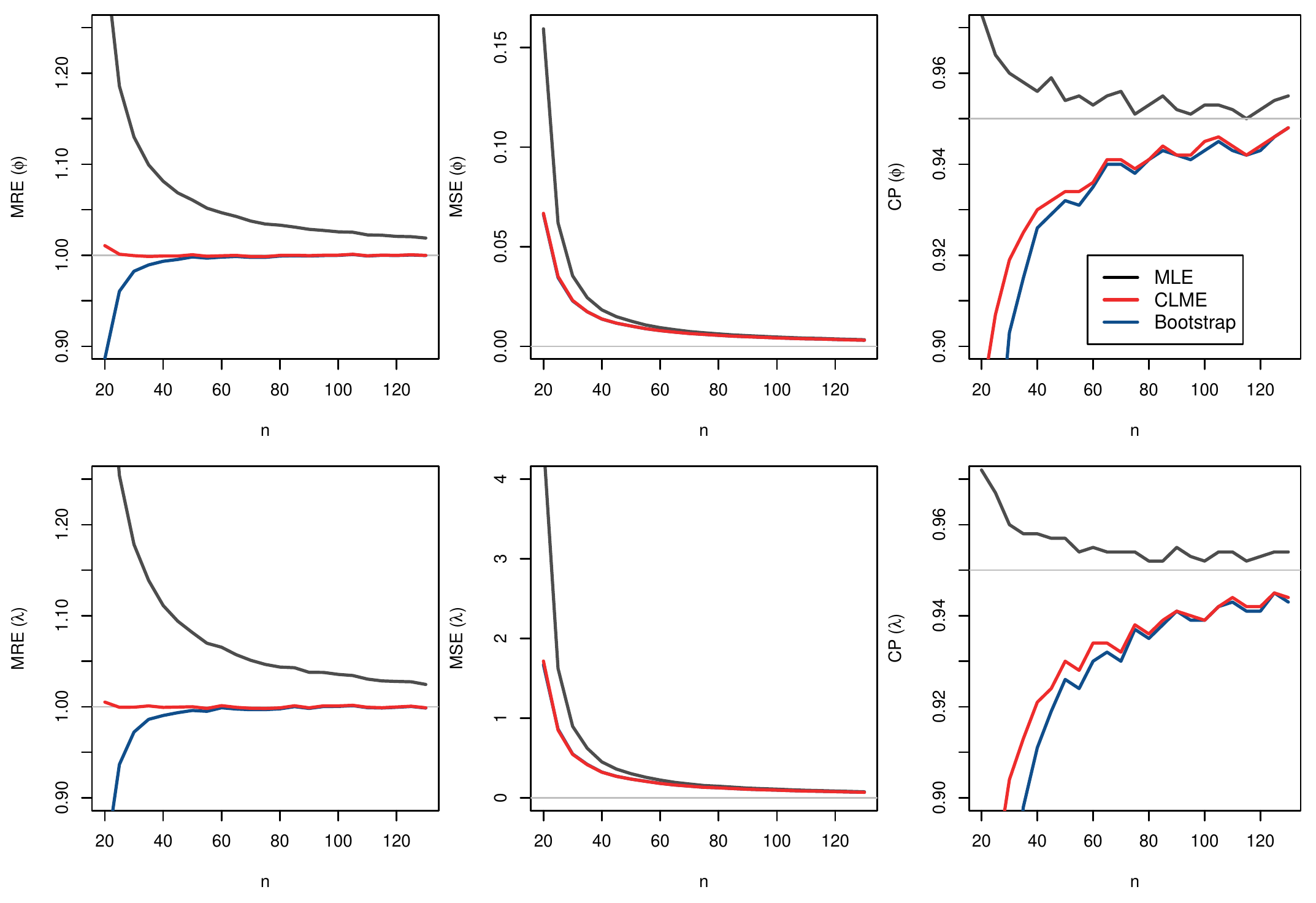}
	\caption{MREs, MSEs related to the estimates of $\lambda=0.5$ and $\phi=2$ for $N=30,000$ simulated samples, considering different values of $n$.}\label{fsimulation2}
\end{figure}

From Figures \ref{fsimulation1} and \ref{fsimulation2}, we observed that the estimates of $\phi$ and $\lambda$ are asymptotically unbiased, i.e., the MREs tend to one when $n$ increases and the RMSEs decrease to zero for $n$ large. The CMLE present superior performance than the bootstrap approach for both parameters for any sample sizes. Taking into account the results of the simulation studies, the maximum likelihood estimators combined with the corrective bias approach discussed in Section \ref{secbiascor} should be considered for estimating the parameters of the IWL distribution.

\subsection{Censored Data}

In this section, we considered the MLES in the presence of random censored data. The censored data is generated following the same procedure presented by Goodman et al. \cite{goodman2006survival}. In our case, we presented two scenarios where we obtained approximately $0.3$ and $0.5$ proportions of censored data, i.e., $30\%$ and $50\%$ of censorship. The simulation study is performed considering $\boldsymbol{\theta}=(2,4)$, $N=2,000$ and $n=(10,15,\ldots$, $130)$ The maximum likelihood estimates were computed using the log-likelihood functions (\ref{logverctr}) with the maxLik package available in R. The solution for the maximum was unique for all initial values.

\begin{figure}[!htb]
	\centering
	\includegraphics[scale=0.7]{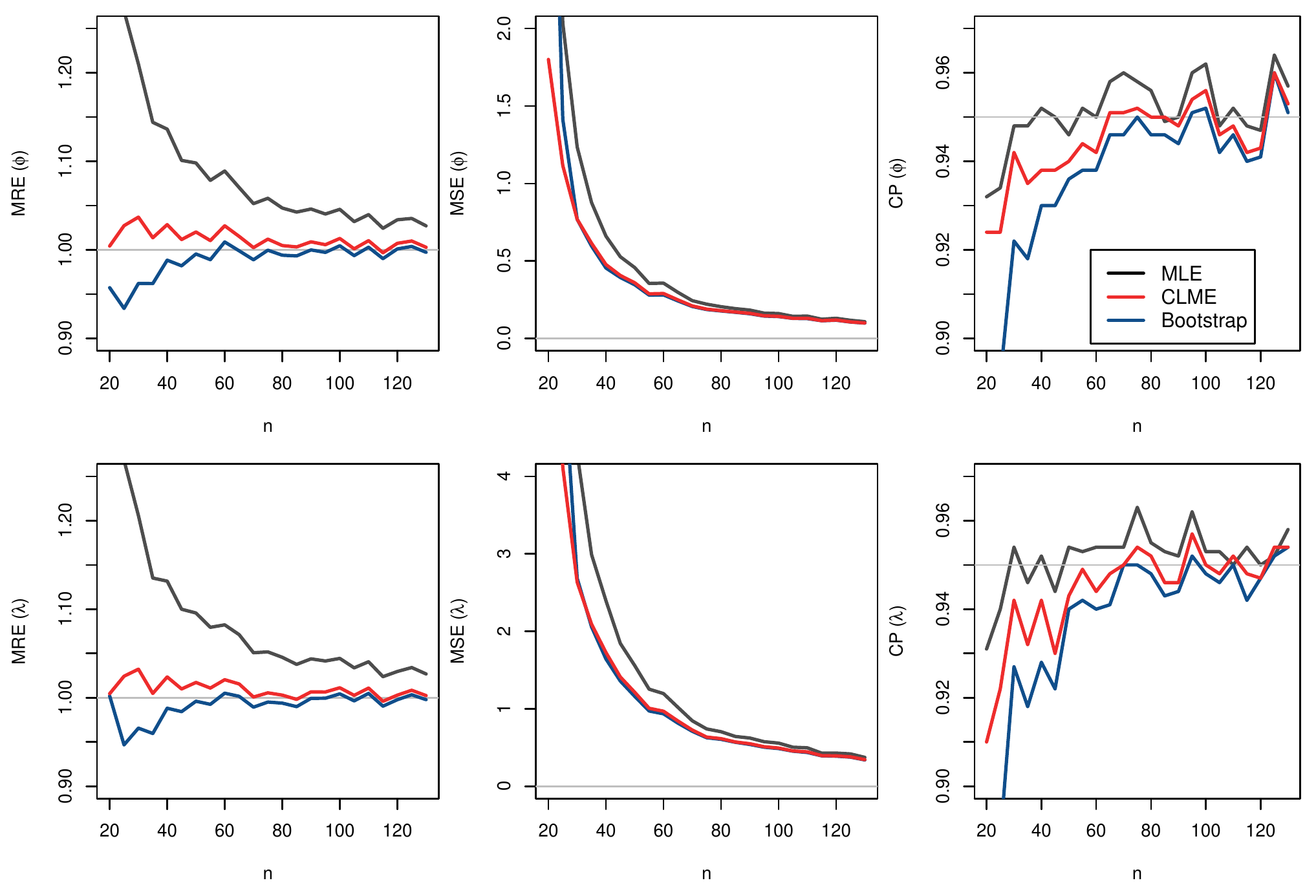}
	\caption{MREs, RMSEs related to the estimates of $\lambda=2$ and $\phi=4$ for $N=10,000$ simulated samples, considering different values of $n$ and $30\%$ of censorship.}\label{fsimulation3}
\end{figure}
\begin{figure}[!htb]
	\centering
	\includegraphics[scale=0.68]{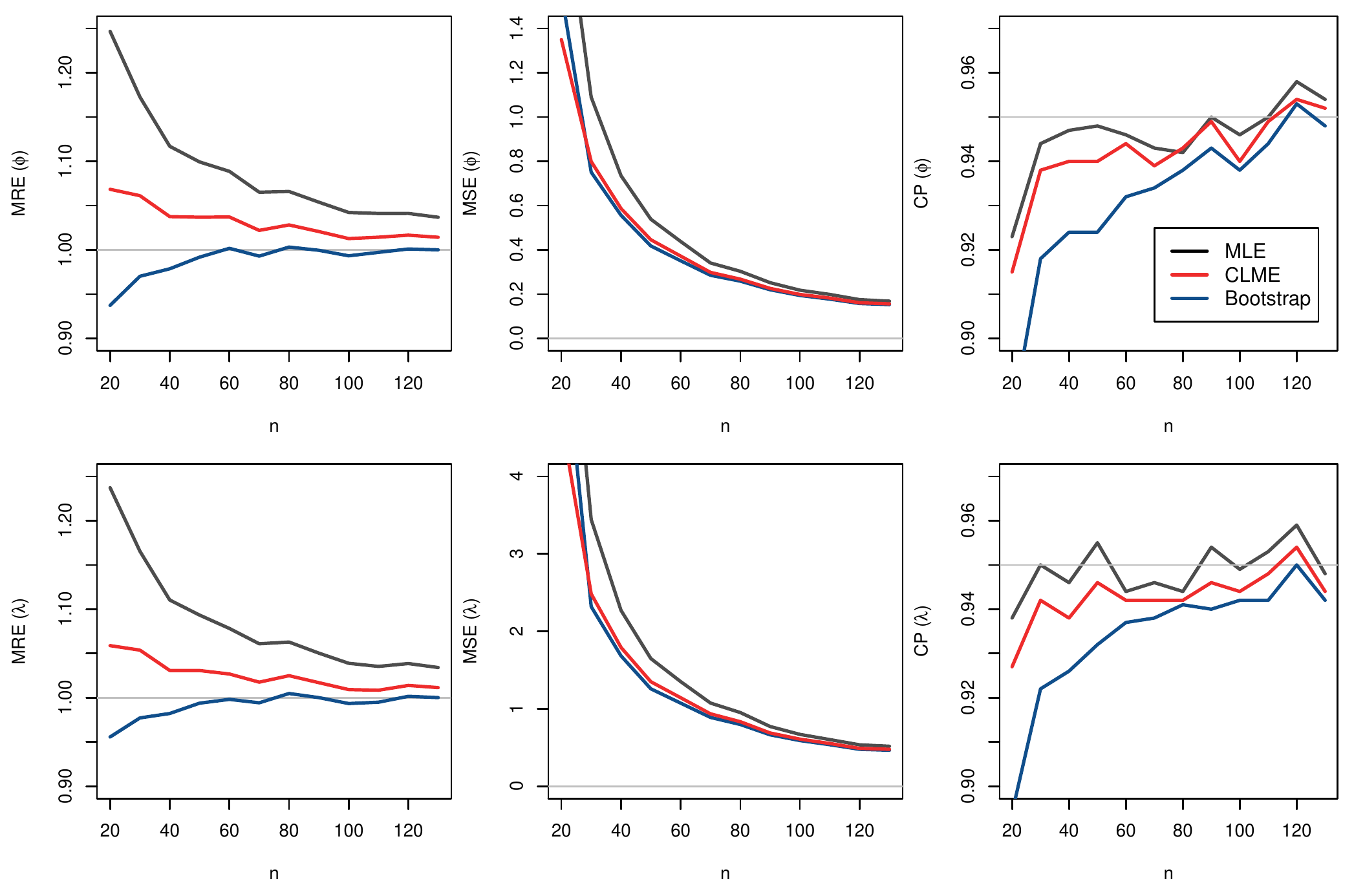}
	\caption{MREs, RMSEs related to the estimates of $\lambda=2$ and $\phi=4$ for $N=2,000$ simulated samples, considering different values of $n$ and $50\%$ of censorship.}\label{fsimulation4}
\end{figure}

Figures \ref{fsimulation3} and \ref{fsimulation4} present the MRE, the RMSE and the coverage probability with a $95\%$ confidence level related to the MLE, CMLE and the bootstrap under different values of $n$.

As shown in Figures \ref{fsimulation3} and \ref{fsimulation4} the proposed ACMLE returned more accurate estimates for both parameters when compared with the bootstrap approach or the MLEs. Taking into account the results of the simulation studies, the approximated corrected bias approach combined with the maximum likelihood estimators should be consider for estimating the parameters of the IWL distribution in the presence of censorship.

\section{Application}

In this section, recall the real data set briefly presented in Section 1. The analyze of the  distribution that better fit the proposed data is relevant to avoid higher costs for the company. Table \ref{tableairplane} presents the data related to failure time of (in days) of $194$ devices in an aircraft (+ indicates the presence of censorship).

\begin{table}[ht]
	\caption{Data set related to the failure time of $194$ devices in an aircraft. \\}
\centering 
	{\begin{tabular}{c c c c c c c c c c c c c c } 
		\hline 
	
		43 & 29 & 37 & 88 &  5 & 14 &  9 & 43+ &  1 & 78 &  1 & 77 & 17 & 100 \\ 
		 3 & 119+ & 22 &  3 &  8 & 80 &  1 & 19 & 157+ & 65 & 34 & 13 & 62+ &  2 \\  
		 1 &  1 & 2 &  3 & 6 &  1 &  2 &  5 &  7 &  6 &  1 &  1 &  4 &  1 \\ 
		1 &  1 &  2 & 7 &  2 &  1 & 1 &  2 &  1 &  1 &  7 &  1 &  1 &  4  \\
		1 &  4 &  2 &  4 & 5 &  5 &  4 &  3 & 2 &  2 &  2 &  3 &  3 &  9  \\
		 1 &  6 &  9 &  2 &  5 & 7 &  4 &  2 &  1 &  2 & 2 &  3 & 11 &  8  \\
		3 &  1 &  2 &  2 &  2 &  2 & 2 &  1 &  3 & 20+ &  8 &  8 & 197 & 20 \\
		14 &  7 & 29 &  7 & 16 & 34 & 25 & 10 & 80 & 42 & 32 &  1 &  3 &  1 \\
		 12 & 7 &  7 & 39+ & 60 & 53 & 32 &  9 & 8 &  1 &  1 & 27 &  2 &  4  \\
		8 & 13 & 7 &  7 &  1 & 19 &  7 & 12 & 19 & 5 & 18 &  1 &  4 & 18  \\
		20 & 9 & 14 & 13 & 70 & 18 &  3 &  7 & 20 &  3& 11 & 10 &  3 & 38+ \\
		278 & 13 & 79 & 145+ & 19 &  2 & 18 &  2 & 65 & 14 & 31 &10 & 19 &  5 \\
		 9 & 45 & 13 &  5 &  1 &  1 & 31 & 35 & 34 &  4 &  3 &  5 & 12 & 140+ \\
		 106 &  5 & 40 & 130+ & 21 & 19 & 7 & 10 & 91 & 193 & 64 & 85+ \\ [0ex] 
		\hline 
	\end{tabular}}\label{tableairplane}
\end{table}

The results obtained from the IWL distribution were compared to the Weibull, Gamma, Lognormal,  Logistic, Inverse Weibull and Inverse Lindley distribution and the nonparametric survival curve adjusted using the Kaplan-Meier estimator. Initially, in order to verify the behavior of the empirical hazard function it will be considered the TTT-plot (total time on test) proposed by Barlow and Campo \cite{barlow}. The TTT-plot is achieved through the consecutive plot of the values $[r/n,G(r/n)]$ where 
$
G(r/n)= \left(\sum_{i=1}^{r}t_i +(n-r)t_{(r)}\right)/ {\sum_{i=1}^{n}t_i}, \quad r=1,\ldots,n, \ i=1,\ldots,n,
$
and $t_{i}$ is the order statistics. If the curve is concave (convex), the hazard function is increasing (decreasing), when it starts convex and then concave (concave and then convex) the hazard function will have a bathtub (inverse bathtub) shape.

Different discrimination criterion methods based on log-likelihood function evaluated at the MLEs were also considered. The discrimination criterion methods are respectively: Akaike information criterion (AIC) computed through $\f{AIC}=-2l(%
\boldsymbol{\hat{\theta}};\boldsymbol{x})+2k$, Corrected Akaike information
criterion $\f{AICC}=\f{AIC}+[{2\,k\,(k+1)}/{(n-k-1)}]$, Hannan-Quinn information
criterion $\f{HQIC}=-2\,l(\boldsymbol{\hat{\theta}};\boldsymbol{x})+2\,k\,\log
\left( \log (n)\right) $ and the consistent Akaike information criterion $%
\f{CAIC=AIC}+k\log(n)-k$, where $k$ is the number of
parameters to be fitted and $\boldsymbol{\hat{\theta}}$ the estimates of $%
\boldsymbol{\theta }$. The best model is the one which provides the minimum values
of those criteria.

Since the data has random censoring mechanism, consequently the equations (\ref{verowl21}) and (\ref{verowl22}) were used to compute the MLEs. Table 2 displays the MLEs, standard-error and  $95\%$ confidence intervals for $\phi$ and $\lambda$. Table \ref{discairplane} presents the results of AIC, AICC, HQIC, CAIC criteria, for different probability distributions.

\begin{table}[ht]
	\caption{MLE, Standard-error and  $95\%$ confidence intervals for $\phi$ and $\lambda$ \\}
	\centering 

	{\begin{tabular}{ c | c |  c| c }
			\hline
			$\boldsymbol{\theta}$  & MLE & SE & $CI_{95\%}(\boldsymbol{\theta})$ \\ \hline
			\ \ $\phi$ \ \   & 0.643 &  0.059 &  (0.527;  0.760)  \\ \hline
			\ \ $\lambda$   \ \  &  2.825 &  0.296 &  (2.245; 3.405) \\ \hline
		\end{tabular}}\label{resairplane} 
\end{table}

\begin{table}[ht]
	\caption{Results of AIC, AICC, HQIC, CAIC criteria for different probability distributions considering the data set related to the failure time of $194$ of devices in an aircraft. \\}
\centering 
{\small
	\begin{tabular}{c|c|c|c|c|c|c|c}
			\hline
			Test   &  IW. Lindley & Weibull & Gamma & Lognormal & Logistic & I Weibull & I Lindley \\ \hline
			AIC   & \textbf{1392.66} & 1452.37 & 1474.44 & 1408.44 & 1818.42 & 1392.70 & 1418.75 \\ 
			AICC   & \textbf{1392.73} & 1452.43 & 1474.50 & 1408.51 & 1818.48 & 1392.76 & 1416.78 \\
			HQIC   & \textbf{1395.31} & 1455.02 & 1477.08 & 1411.09 & 1821.06 &  1395.34 & 1418.08 \\
			CAIC   & \textbf{1401.20} & 1460.91 & 1482.97 & 1416.98 & 1826.95 & 1401.23 & 1427.29 \\ \hline
		\end{tabular}}\label{discairplane}
\end{table}

Figure \ref{grafico-obscajust1} presents the TTT-plot, the survival function adjusted by different distributions and the Kaplan-Meier estimator and the hazard function adjusted by the IWL distribution.

\begin{figure}[!htb]
	\centering
	\includegraphics[scale=0.5]{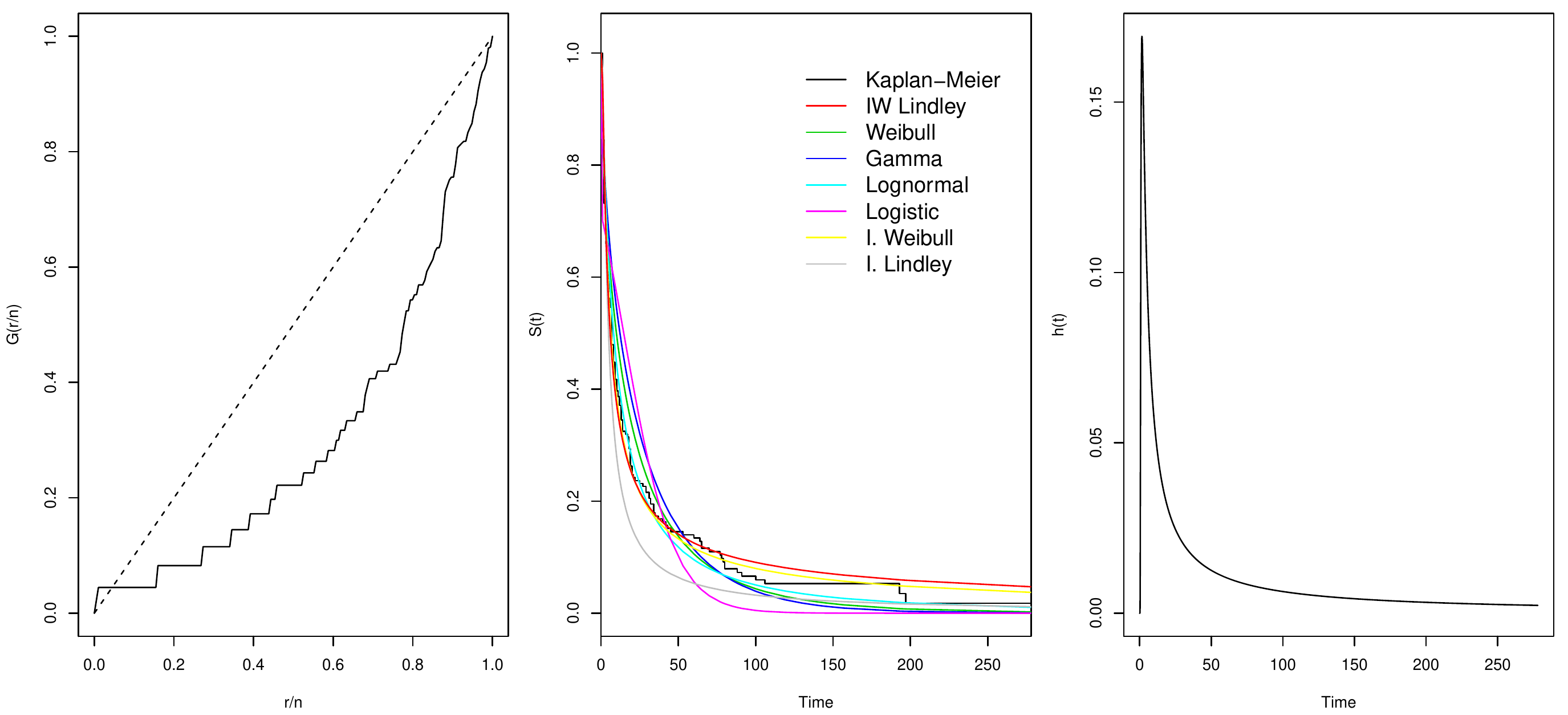}
	\caption{TTT-plot, survival function adjusted by different distributions and the Kaplan–Meier estimator and the hazard function adjusted by the IWL distribution considering data set related to the failure time of $194$ devices.}\label{grafico-obscajust1}
\end{figure}

Comparing the empirical survival function with the adjusted models we observed a goodness of  the fit for the inverse weighted Lindley distribution. This result is also confirmed by the different discrimination criterion methods considered since IWL distribution has the minimum value. Based on the TTT-plot there is an indication that the hazard function has upside-down bathtub failure rate this result is confirmed by the adjusted hazard function. Therefore, from the proposed methodology the data related to the failure time of $194$ devices in an aircraft can be described by the inverse weighted Lindley distribution.

\section{Concluding Remarks}

In this paper, a new distribution called inverse weighted Lindley is proposed and its mathematical properties were studied in detail. The maximum likelihood estimators of the parameters and their asymptotic properties were obtained, we also presented two corrective approaches to derive a modified MLEs that are bias-free to second order, as well as the MLEs in the presence of randomly censored data. The simulation study showed that the CMLE and ACLME present extremely efficient estimators for both parameters for any sample sizes. The practical importance of the IWL distribution was reported in a real application, in which our new distribution returned better fitting in comparison with other well-known distributions. 

\section*{Acknowledgements}

The authors are very grateful to the reviewers for their helpful and useful comments that improved the manuscript. The authors' researchers are partially supported by the Brazilian Institutions: CNPq, CAPES and FAPESP.

\bibliographystyle{tfs}

\bibliography{reference}

\end{document}